\documentclass [a4paper, 12pt]{amsart}

\usepackage[margin=2.3cm]{geometry} 
\usepackage[alphabetic]{amsrefs}
\usepackage{times}
\usepackage{wasysym,stmaryrd,enumerate}

\usepackage{amsmath}
\usepackage{amsfonts}
\usepackage{amssymb}
\usepackage{amsthm}
\usepackage{graphicx}
\usepackage{pb-diagram}
\usepackage{epstopdf}
\usepackage{fancyhdr}
\usepackage[all]{xy}

\newtheorem{cor}{Corollary}[section]
\newtheorem{te}[cor]{Theorem}
\newtheorem{p}[cor]{Proposition}
\newtheorem{quest}[cor]{Question}

\theoremstyle{definition}
\newtheorem{de}[cor]{Definition}

\theoremstyle{remark}
\newtheorem{ob}[cor]{Observation}
\newtheorem{ex}[cor]{Example}

\newcommand{\cz}{\mathbb{C}}
\newcommand{\nz}{\mathbb{N}}
\newcommand{\zz}{\mathbb{Z}}

\newcommand{\rz}{\mathbb{R}}

\newcommand{\nr}{\mathcal{N}}
\newcommand{\unit}{\mathcal{U}}

\newcommand{\vp}{\varphi}
\newcommand{\ve}{\varepsilon}

   {\begin{verse}%
     \small }%
   {\end{verse}}

\DefineSimpleKey{bib}{arx}
\BibSpec{arx}{%
  +{}{\PrintAuthors} {author}
  +{,}{ \textit} {title}
  +{}{ \parenthesize} {date}
  +{,}{ } {note}
  +{.}{ } {transition}
}

\begin{document}

\title{Unitaries in ultraproduct of matrices}

\author{Liviu P\u aunescu}
\address[L. P\u aunescu]{Institute of Mathematics of the Romanian Academy, 21 Calea Grivitei Street, 010702 Bucharest, Romania}
\email{liviu.paunescu@imar.ro}
\thanks{This work was partially supported by a grant of the Romanian National Authority for Scientific Research and Innovation, CNCS - UEFISCDI, project number PN-II-RU-TE-2014-4-0669. A part of this work was done while staying at the Erwin Schrödinger Institute for Mathematics and Physics under the Measured Group Theory program in 2016, supported by European Research Council (ERC) grant of Goulnara Arzhantseva, grant agreement no 259527}

\begin{abstract}
We study various conditions under which a unitary in an ultraproduct of matrices is conjugated to an ultraproduct of permutations.
\end{abstract}

\maketitle

We know that two unitary matrices $u,v\in M_n(\cz)$ are conjugated by a unitary if and only if they have the same eigenvalues, multiplicities included. Let $P_n\subset M_n$ be the subgroup of permutation matrices, i.e. matrices that have exactly one entry of one on each row and column. If $c\in P_n$ corresponds to a cycle of length $n$, then the eigenvalues of $c$ are precisely the $n$ roots of unity, each considered with multiplicity one. If $p\in P_n$ is a general permutation matrix, one has to consider its cycle decomposition. Then its set of eigenvalues is the disjoint union of the sets of eigenvalues corresponding to each cycle. All in all, describing unitary matrices that are conjugated to a permutation is easy and only requires a look at eigenvalues with multiplicity.

What about a unitary in the ultraproduct $u\in\Pi_{k\to\omega}M_{n_k}$? Assume that $u$ is of infinite order. It is easy to see that the spectrum of an element in $\Pi_{k\to\omega}P_{n_k}$ of infinite order is the whole circle $S^1=\{\lambda\in\cz:|\lambda|=1\}$. However this condition is not sufficient to deduce that $u$ is conjugated to an ultraproduct of permutations. The spectrum must also be equally distributed on the circle, in some sense. This can be formalised by introducing a measure on $S^1$, using the trace of spectral projections of $u$, see Section \ref{measure on spectrum}. Alternatively, we can study the momentums of this measure, i.e. the sequence $\big(Tr(u^k)\big)_{k\in\nz^*}$. This will be investigated in Section \ref{traces of powers}.

\section{Introduction}

We use the same notation as in \cite{Pa}. As mentioned $P_n\subset M_n(\cz)$ is the subgroup of permutation matrices inside the algebra of $n$-dimensional matrices. The \emph{trace norm} on $M_n(\cz)$ is defined as $||A||_2=\sqrt{\frac1n\sum_{i,j}|a_{ij}|^2}$, such that $||Id||_2=1$ independent of the dimension. This formula can be also expressed as $||A||_2=\sqrt{Tr(A^*A)}$, where $Tr(A)=\frac1n\sum_ia_{ii}$.

In order to construct the ultraproduct of matrices, fix $\omega$ a free ultrafilter on $\nz$, and fix $(n_k)_k\subset\nz^*$ a sequence such that $\lim_kn_k=\infty$. Then $\Pi_{k\to\omega}M_{n_k}=l^\infty(\nz,M_{n_k})/\nr_\omega$, where $l^\infty(\nz,M_{n_k})$ is the algebra of bounded sequences of matrices in the operatorial norm, and $\nr_\omega=\{(x_k)\in l^\infty(\nz,M_{n_k}): \lim_{k\to\omega}||x_k||_2=0\}$. Also $\Pi_{k\to\omega}P_{n_k}\subset\Pi_{k\to\omega}M_{n_k}$ is the subgroup of elements that have a representative obtained using only permutation matrices. If $x_k\in M_{n_k}$, then $\Pi_{k\to\omega}x_k$ denotes the corresponding element in $\Pi_{k\to\omega}M_{n_k}$.

Let $p\in P_n$. For $i\in\nz^*$ denote by $cyc_i(p)$ the number of points that are part of cycles of length $i$ in $p$ divided by $n$. One can immediately check the formulas: $\sum_{i\in\nz^*}cyc_i(p)=1$ and $Tr(p^n)=\sum_{i|n}cyc_i(p)$. We shall use these equalities extensively. 

For an element $p=\Pi_{k\to\omega}p_k\in\Pi_{k\to\omega}P_{n_k}$, the numbers $cyc_i(p)=\lim_kcyc_i(p_k)$ are well defined for any $i\in\nz^*$. This time $\sum_{i\in\nz^*}cyc_i(p)\leqslant 1$. Define $cyc_\infty(p)=1-\sum_{i\in\nz^*}cyc_i(p)$. These numbers were introduced in \cite{El-Sz}, where the following theorem have been proven.

\begin{p}(Proposition 2.3(4),\cite{El-Sz})
Two elements $p,q\in\Pi_{k\to\omega}P_{n_k}$ are conjugated if and only if $cyc_i(p)=cyc_i(q)$ for any $i\in\nz^*$.
\end{p}

\section{A consequence of unique embedability of amenable algebras}

We write down the first question that we have in mind in this paper. Denote by $\unit(A)$ the set of unitaries of a unital $*$-algebra, i.e. $\unit(A)=\{u\in A:u^*u=Id=uu^*\}$. We write $\unit(n)$ instead of $\unit(M_n(\cz))$.

\begin{quest}\label{main question}
When is a unitary $u\in\unit(\Pi_{k\to\omega}M_{n_k})$ conjugated to an element of $\Pi_{k\to\omega}P_{n_k}$?
\end{quest}

The theory of von Neumann algebras immediately provides a partial answer.

\begin{p}\label{conjugated hyperfinite}
Let $N$ be a hyperfinite von Neumann algebra and $\Theta_1,\Theta_2:N\to\Pi_{k\to\omega}M_{n_k}$ be two trace preserving embeddings. Then there exists $u\in\unit(\Pi_{k\to\omega}M_{n_k})$ such that $\Theta_2=Adu\circ\Theta_1$.
\end{p}

This is a classic fact. In order to prove it one has to do it first for finite dimensional algebras, and then use the hyperfine property of $N$ in conjunction with a diagonal argument made available by the presence of the ultraproduct. 

\begin{cor}\label{conjugated}
Let $u,v\in\unit(\Pi_{k\to\omega}M_{n_k})$. Then $u$ and $v$ are conjugated in $\Pi_{k\to\omega}M_{n_k}$ if and only if $Tr(u^k)=Tr(v^k)$ for any $k\in\zz$.
\end{cor}
\begin{proof}
Let $A_u$ and $A_v$ be the von Neumann algebras generated by $u$ and $v$ respectively in $\Pi_{k\to\omega}M_{n_k}$. Then the map $\vp:\{u\}\to\{v\}$ extends to an isomorphism from $A_u$ to $A_v$, as $\vp$ is trace preserving on a dense subset. These algebras are hyperfinite, because they are abelian. So $A_u$ and $A_v$ are two embeddings of the same hyperfinite algebra. By Proposition \ref{conjugated hyperfinite}, these embeddings are conjugated. It follows that $u$ and $v$ are conjugated.
\end{proof}

Coming back to Question \ref{main question}, we get the following characterisation.

\begin{cor}\label{same powers}
An element $u\in\unit(\Pi_{k\to\omega}M_{n_k})$ is conjugated to an ultraproduct of permutations if and only if there exists $p\in\Pi_{k\to\omega}P_{n_k}$ such that $Tr(u^k)=Tr(p^k)$.
\end{cor}

So fix $p\in\Pi_{k\to\omega}P_{n_k}$, and consider the sequence $(Tr(p^k))_k$. What properties does this sequence have, properties that $(Tr(u^k))_k$, for $u\in\unit(\Pi_{k\to\omega}M_{n_k})$ doesn't have in general? Trivially, $Tr(u^k)$ should be a positive real for any $k$. A first non-trivial property is that $Tr(p^2)\geqslant Tr(p)$. This is because $Tr(p^2)=cyc_2(p)+cyc_1(p)=cyc_2(p)+Tr(p)$. One can easily construct $u\in\unit(\Pi_{k\to\omega}M_{n_k})$ with $Tr(u^k)\in\rz_+$ and $Tr(u^2)<Tr(u)$. The equality $Tr(p^3)=cyc_3(p)+Tr(p)$ yields a similar condition, or replace $3$ by any prime number. The next one, $Tr(p^4)=cyc_4(p)+cyc_2(p)+cyc_1(p)$, only gives $Tr(p^4)\geqslant Tr(p^2)$, which is just the first condition for $p^2$ instead of $p$. And so we arrive to the equality $Tr(p^6)=cyc_6(p)+cyc_3(p)+cyc_2(p)+cyc_1(p)$. This can be rewritten as $Tr(p^6)\geqslant Tr(p^3)+Tr(p^2)-Tr(p)$, but as we can see the situation gets messier. This we investigate in the next section.

\section{Describing traces of powers of permutations}\label{traces of powers}

Fix an element $p\in\Pi_{k\to\omega}P_{n_k}$. Remember that $Tr(p^k)=\sum_{i|k}cyc_i(p)$. We want to reverse this system of equations.

\begin{p}\label{inclusion-exclusion}
Let $i\in\nz^*$. If $i=a_1^{r_1}\ldots a_t^{r_t}$ is the prime decomposition of $i$ then:
\[cyc_i(p)=\sum_{(\ve_1,\ldots,\ve_t)\in(0,1)^t}(-1)^{\ve_1+\ldots+\ve_t}Tr(p^{a_1^{r_1-\ve_1}\ldots a_t^{r_t-\ve_t}}).\]
\end{p}
\begin{proof}
This is an inclusion-exclusion principle. We just verify the equality. On the righthand side we only have terms of the type $cyc_j(p)$ with $j|i$. The term $cyc_i(p)$ only appears once, when $\ve_1=\ve_2=\ldots=\ve_t=0$. We have to verify that everything else cancels out.

Fix $j=a_1^{s_1}\ldots a_t^{s_t}$, a strict divisor of $i$. We can assume, without the loss of generality, that $s_k<r_k$ for $k=1,\ldots,v$ and $s_k=r_k$ for $k=v+1,\ldots,t$. In order for $cyc_j(p)$ to show up in $Tr(p^{a_1^{r_1-\ve_1}\ldots a_t^{r_t-\ve_t}})$, we must have $\ve_{v+1}=\ldots=\ve_t=0$. Then $cyc_j(p)$ appears exactly once for each choice $(\ve_1,\ldots,\ve_v)\in(0,1)^v$. Hence, the coefficient of $cyc_j(p)$ in the righthand side of the equality is:
\[\sum_{(\ve_1,\ldots,\ve_v)\in(0,1)^v}(-1)^{\ve_1+\ldots+\ve_v}=0.\]
For this last equality, it is important that $v>0$, which is true since $j<i$.
\end{proof}

The equation from this last proposition can be used to define numbers $cyc_i(u)$ for any $u\in\unit(\Pi_{k\to\omega}M_{n_k})$. We now show that $u$ is conjugated to an ultraproduct of permutations if and only if these numbers are positive.

\begin{p}\label{exists permutation}
Let $(c_i)_{i\in\nz^*}$ be a sequence of real numbers. Then there exists $p\in\Pi_{k\to\omega}P_{n_k}$ such that $cyc_i(p)=c_i$ for any $i\in\nz^*$ if and only if the following conditions hold:
\begin{enumerate}
\item $c_i\geqslant 0$, for any $i\in\nz^*$;
\item $\sum_{i>0}c_i\leqslant 1$.
\end{enumerate} 
\end{p}
\begin{proof}
Once you understand how numbers $cyc_i(p)$ work, this is an easy statement. If such an element $p\in\Pi_{k\to\omega}P_{n_k}$ exists then, by definition, $cyc_i(p)\geqslant 0$ and $\sum_{i>0}cyc_i(p)\leqslant 1$.

For the reverse implication, define $n(i,k)=i\cdot\lfloor \frac{c_i\cdot n_k}i\rfloor$. Then $\sum_i n(i,k)\leqslant(\sum_i c_i)\cdot n_k\leqslant n_k$. We construct a permutation $p_k\in P_{n_k}$, having exactly $n(i,k)/i$ cycles of length $i$. Then $cyc_i(p_k)=n(i,k)/n_k$.

Let $p=\Pi_{k\to\omega}p_k$. Then:
\[cyc_i(p)=\lim_{k\to\omega}\frac{i\cdot\lfloor \frac{c_i\cdot n_k}i\rfloor}{n_k}=c_i.\]
The last equality follows from the fact that $\lim_kn_k=\infty$.
\end{proof}

Now we can put everything together to prove the following result.

\begin{te}
Let $u\in\unit(\Pi_{k\to\omega}M_{n_k})$. Construct $c_i=\sum_{(\ve_1,\ldots,\ve_t)\in(0,1)^t}(-1)^{\ve_1+\ldots+\ve_t}Tr(u^{a_1^{r_1-\ve_1}\ldots a_t^{r_t-\ve_t}})$. Then $u$ is conjugated to an element of $\Pi_{k\to\omega}P_{n_k}$ if and only if $c_i\geqslant 0$ for any $i\in\nz^*$.
\end{te}
\begin{proof}
Our first goal is to prove that $\sum_{i>0}c_i\leqslant 1$, in order to apply Proposition \ref{exists permutation}. One has to reverse Proposition \ref{inclusion-exclusion} to deduce that $Tr(u^i)=\sum_{j|i}c_j$ (define $tr(i)=\sum_{j|i}c_j$, apply Proposition \ref{inclusion-exclusion} to these numbers, and use induction to deduce that $Tr(u^i)=tr(i)$).

Then $\sum_{j\leqslant i}c_j\leqslant Tr(u^{i!})\leqslant 1$. As $\sum_ic_i=\lim_i(\sum_{j\leqslant i}c_j)$, it follows that $\sum_{i>0}c_i\leqslant 1$. Now, we can apply Proposition \ref{exists permutation}, to deduce that there is $p\in\Pi_{k\to\omega}P_{n_k}$ such that $cyc_i(p)=c_i$. Then $Tr(u^i)=Tr(p^i)$ for any $i>0$. Use Corollary \ref{conjugated} to deduce that $u$ and $p$ are conjugated by a unitary in $\Pi_{k\to\omega}M_{n_k}$.
\end{proof}

\section{A measure on the spectrum}\label{measure on spectrum}

As we said, the sequence $\big(Tr(u^k)\big)_k$ hides the existence of a measure on the spectrum of $u$. This object is the equivalent of multiplicity of eigenvalues in the finite dimensional case. As the spectrum of any unitary is a subset of $S^1$, we consider these measures as being defined on $S^1$. It is well known that a Borel measure on the circle is completely determined by its moments.

\begin{te}\label{equal moments}
If $\mu$ and $\lambda$ are two real Borel measures on the unit circle such that $\int_{S^1}z^nd\mu=\int_{S^1}z^nd\lambda$ for every $n\in\nz$, then $\mu=\lambda$.
\end{te}
\begin{proof}
Let $\theta=\mu-\lambda$. Then $\theta$ is a bounded variation real Borel measure. Moreover, its Fourier coefficients are: $\int_{S^1}z^nd\theta=0$ for any $n\in\nz$. By the Theorem of brothers Frigyes Riesz and Marcel Riesz, it follows that $\theta$ is absolutely continuous with respect to the Lebesgue measure, denoted by $dz$ in this proof. It follows that $\theta=f\cdot dz$ for some real valued function $f\in L^1(S^1,dz)$. The integral equality is now $\int_{S^1}fz^ndz=0$ for $n\in\nz$. By conjugation, using the fact that $f$ is real valued, the above equality holds for any $n\in\zz$. As $f\in L^1(S^1,dz)$, by the unicity of Fourier coefficients for functions in $L^1$, it follows that $f=0$.
\end{proof}

We construct our measure using Borel calculus. Let $(M,Tr)$ be a tracial von Neumann algebra and fix $u\in M$ a normal element, i.e. $uu^*=u^*u$. Let $A(u)$ be the von Neumann subalgebra generated by $u$ inside $M$. As $u$ is normal, $A(u)$ is abelian. It follows that $\big(A(u),Tr\big)\simeq L^\infty(\sigma(u),\mu_u)$, such that $u$ corresponds to the function $z\to z$ in $L^\infty(\sigma(u))$.

It is this measure $\mu_u$ that we are interested in. First of all, lets compute its moments. By definition, $\int_{S^1}z^nd\mu_u=Tr(u^n)$, so indeed this is the measure that we are looking for. Combining Corollary \ref{conjugated} and Theorem \ref{equal moments}, we get the following statement.

\begin{p}
Two unitaries $u,v\in\Pi_{k\to\omega}M_{n_k}$ are conjugated if and only if their associated measures are equal.
\end{p}

\subsection{The associated measure of an ultraproduct of permutations} In this section we characterise $\mu_p$ for $p\in\Pi_{k\to\omega}P_{n_k}$.

For $z\in S^1$, denote by $\delta_z$, the Dirac measure, i.e. $\delta_z(A)=\chi_A(z)$. For $i\in\nz^*$, define \[\lambda_i=\frac1i\sum_{z^i=1}\delta_z.\]
Then $\int_{S^1}z^nd\lambda_i=1$ iff $i|n$, otherwise the integral is equal to $0$. Denote by $\lambda_\infty$ the probabilistic Lebegue measure on $S^1$. In this case $\int_{S^1}z^nd\lambda_\infty=0$ for $n>0$.

Let $p\in\Pi_{k\to\omega}P_{n_k}$. Construct \[\mu=\sum_{i\in\nz^*\cup\{\infty\}}cyc_i(p)\cdot\lambda_i.\]
Then $\int_{S^1}z^nd\mu=\sum_{i\in\nz^*\cup\{\infty\}}\big(cyc_i(p)\cdot\int_{S^1}z^nd\lambda_i\big)=\sum_{i|n}cyc_i(p)=Tr(p^n)$. By Theorem \ref{equal moments}, it follows that $\mu=\mu_p$. We proved the following theorem.

\begin{te}
A unitary $u\in\Pi_{k\to\omega}M_{n_k}$ is conjugated to an element in $\Pi_{k\to\omega}P_{n_k}$ if and only if its associated measure $\mu_u$ is in the closed convex hull of the set $\big\{\lambda_i:i\in\nz^*\cup\{\infty\}\big\}$.
\end{te}
\begin{proof}
We need a proof just for the "closed convex hull" part. The role of this closure is to construct infinite convex combinations, i.e. measures $\sum_ic_i\lambda_i$ with an infinite number of non-zero scalars $c_i$. The strong topology does this job. So $\mu_j\to_j\mu$ if $\mu_j(A)\to_j\mu(A)$ for any Borel set $A$.

Let $M$ be the convex hull of the set $\big\{\lambda_i:i\in\nz^*\cup\{\infty\}\big\}$, that is finite convex combinations. The only non trivial part is to prove that a measure $\mu$ in the closure of this set is still a measure of the type $\sum_ic_i\lambda_i$ with $\sum_ic_i=1$.

Let $\mu_j\in M$ be such that $\mu_j\to_j\mu$ in the strong topology. Let $\mu_j=\sum_ic_i^j\lambda_i$. As we are working only with probabilistic positive measures, it follows that $\int fd\mu_j\to_j\int fd\mu$ for any bounded Borel function $f$. This implies, in particular, convergence of moments. So $\int z^nd\mu=\lim_j\int z^nd\mu_j=\lim_j(\sum_{i|n}c_i^j)$. Using Propostion \ref{inclusion-exclusion} we can deduce that $\lim_jc_i^j$ is a finite sum of terms of the type $\int z^sd\mu$, so it has to exists. Let $c_i=\lim_jc_i^j$. Let $c_\infty=1-\sum_{i\in\nz^*}c_i$. Construct $\mu_1=\sum_{i\in\nz^*\cup\{\infty\}}c_i\lambda_i$. It can be verified that $\mu$ and $\mu_1$ have the same moments, so by Theorem \ref{equal moments} they have to be equal. 
\end{proof}

In the introduction we said that deciding whenever a unitary matrix is conjugated to a permutation one, only requires a look at its eigenvalues with multiplicity. The last theorem is an analogue of this statement for ultraproducts, with spectrum instead of eigenvalues and the associated measure instead of multiplicities.

\section{Application to hyperfinite representations}

Let us recall the notions of hyperfinite and sofic groups.

\begin{de}
A group $G$ is called \emph{hyperfinite} if there exists an injective morphism from $G$ to $\unit(\Pi_{k\to\omega}M_{n_k})$. A group $G$ is called \emph{sofic} if there exists an injective morphism from $G$ to $\Pi_{k\to\omega}P_{n_k}$.
\end{de}

Clearly every sofic group is hyperfinite, while the converse is a major open problem of the theory. For hyperlinear groups we have the following theorem.

\begin{te}(\cite{Ra}, Proposition 2.5)
For a hyperlinear group $G$, there exists a morphism $\phi:G\to\unit(\Pi_{k\to\omega}M_{n_k})$ such that $Tr(\phi(g))=0$ for any $g\neq e$.
\end{te}

It follows that $Tr(\phi(g)^n)\in\{0,1\}$. By Corollary \ref{same powers}, $\phi(g)$ is conjugated to an element in $\Pi_{k\to\omega}P_{n_k}$ for any $g\in G$. Unfortunately, this doesn't mean that all $\phi(g)$ are conjugated at the same time to elements of $\Pi_{k\to\omega}P_{n_k}$, so this Corollary does not imply that every hyperlinear group is sofic. All that we have here is a hyperlinear representation of a group such that, individually, each element in the image is conjugated to a permutation. 

However, there are some questions to be asked here. Define:
\[PP_n=\{u\in\unit(n):u=vpv^*,\mbox{ where }v\in\unit(n) \mbox{ and }p\in P_n\}.\]
What are the subgroups of this set $PP_n$? Of course, $P_n$ and its conjugates are subgroups here, but are there some others? As we saw, any hyperlinear group is a subgroup in $\Pi_{k\to\omega}PP_{n_k}$. Does this observation imply that any hyperlinear group is a subgroup in $\Pi_{k\to\omega}H_k$, where each $H_k$ is a subgroup in $PP_{n_k}$. We present an example to illustrate some aspects of this problem.

\begin{ex}
Let $n$ be an odd number. Denote by $c\in P_n$ the matrix $c(i,j)=\delta_i^{j+1}$, and $d\in D_n$ the matrix $d(i,j)=\delta_i^j\lambda^i$, where $\lambda$ is the first root of unity of order $n$, $\lambda\neq 1$. Then $c$ is a permutation matrix, and $d$ is conjugated to a permutation matrix, so $c,d\in PP_n$. We now prove that also $cd\in PP_n$. By matrix multiplication $cd(i,j)=\delta_i^{j+1}\lambda^{i}$. The characteristic polynomial of $cd$: 
\[det(X\cdot Id-cd)=X^n+(-1)\cdot(-\lambda)\cdot\ldots\cdot(-\lambda^{n-1})=X^n-1.\]

A matrix is in $PP_n$ if and only if its characteristic polynomial is $X^n-1$, so indeed $cd\in PP_n$. 
On the other side $[c,d]=cdc^{-1}d^{-1}=\lambda^{-1}\cdot Id$, so this group is not included in $PP_n$. This fact also implies that there is no $v\in\unit(n)$, such that $c,d\in vP_nv^*$, i.e. $c$ and $d$ are not simultaneously permutation matrices.

From now on, we denote by $c_n$ and $d_n$ the matrices that we constructed, as the dimension will play a role. Fix $(n_k)_k$ an increasing sequence of natural numbers.  Define a morphism $\phi:\zz^2\to\Pi_{k\to\omega}M_{n_k}$ by $\phi(c)=\Pi_{k\to\omega}c_{n_k}$ and $\phi(d)=\Pi_{k\to\omega}d_{n_k}$. The strange thing about this morphism is that it is well-defined, i.e. $\phi(c)$ and $\phi(d)$ commute. This is because $\lambda_n\to_n1$, where $\lambda_n$ is the first root of unity different then $1$.

It is easy to see that $Tr(\phi(z))=0$ for any $z\in\zz$, $z\neq (0,0)$, so $\phi$ is a hyperlinear representation of $\zz^2$. By Propostion \ref{conjugated hyperfinite}, $\phi$ is conjugated to a sofic representation of $\zz^2$. It follows that we can find some representatives $\phi(c)=\Pi_{k\to\omega}c^1_{n_k}$ and $\phi(d)=\Pi_{k\to\omega}d^1_{n_k}$, such that the group generated by $c^1_{n_k}$ and $d^1_{n_k}$ is included in $PP_{n_k}$ for any $k$.
\end{ex}

\begin{ob}
In the above example, if $\lambda_n$ is taken to be a primitive root of unity closer to $-1$, and if this root of unity is used in the definition of $d_n$, then $c$ and $d$ no longer commute. One obtains a hyperlinear representation of the group $G=\zz\ltimes H$, where $\zz$ is generated by $\phi(c)$ and $H$, a subgroup of $\Pi_{k\to\omega}D_{n_k}$, isomorphic to $\zz^2$. This group is still amenable, in fact it is metabelian, so the conclusion of the example still holds.
\end{ob}

\end{document}